\documentclass[12pt]{article}

\usepackage[utf8]{inputenc}

\usepackage[a4paper,nomarginpar]{geometry}
\usepackage[english]{babel}

\usepackage{amsmath,amsfonts,amssymb,amsthm}
\usepackage{amscd}
\usepackage{tikz}

\theoremstyle{theorem}
\newtheorem{theorem}{Theorem}
\newtheorem{lemma}[theorem]{Lemma}
\newtheorem{proposition}[theorem]{Proposition}
\newtheorem{corollary}[theorem]{Corollary}

\theoremstyle{definition}
\newtheorem{definition}[theorem]{Definition}
\newtheorem{example}[theorem]{Example}
\newtheorem{remark}[theorem]{Remark}

\theoremstyle{remark}
\theoremstyle{question}
\theoremstyle{example}

\newtheorem{question}[theorem]{Question}

\newcommand{\N}{\mathbb{N}}
\newcommand{\Z}{\mathbb{Z}}
\newcommand{\Q}{\mathbb{Q}}
\newcommand{\R}{\mathbb{R}}
\newcommand{\C}{\mathbb{C}}

\newcommand{\K}{\mathbb{K}}

\newcommand{\cT}{\mathcal{T}}

\newcommand{\norm}[1]{\lVert#1\rVert}

\newcommand{\Orb}   {\operatorname{Orb}}
\newcommand{\card}  {\operatorname{card}}
\newcommand{\ldens} {\operatorname{\underline{dens}}}
\newcommand{\udens} {\operatorname{\overline{dens}}}

\newcommand{\spa}   {\operatorname{span}}

\newcommand{\re}    {\operatorname{Re}}

\newcommand{\eps}{\varepsilon}
\newcommand{\ov} {\overline}

\newcommand{\shs}{\hspace*{6mm}} 

\author{N. C. Bernardes Jr.,  
A. Bonilla  and 
A. Peris
}

\date{} 

\title{Mean Li-Yorke chaos in Banach spaces}

\begin{document}

\maketitle

\begin{abstract}
We investigate the notion of mean Li-Yorke chaos for operators on Banach
spaces. We show that it differs from the notion of distributional
chaos of type 2, contrary to what happens in the context of topological
dynamics on compact metric spaces. We prove that an operator is mean
Li-Yorke chaotic if and only if it has an absolutely mean irregular vector.
As a consequence, absolutely Ces\`aro bounded operators are never mean
Li-Yorke chaotic. Dense mean Li-Yorke chaos is shown to be equivalent to
the existence of a dense (or residual) set of absolutely mean irregular
vectors. As a consequence, every mean Li-Yorke chaotic operator
is densely mean Li-Yorke chaotic on some infinite-dimensional closed
invariant subspace. A (Dense) Mean Li-Yorke Chaos Criterion and a sufficient
condition for the existence of a dense absolutely mean irregular manifold
are also obtained. Moreover, we construct an example of an invertible
hypercyclic operator $T$ such that every nonzero vector is absolutely mean
irregular for both $T$ and $T^{-1}$. Several other examples are also
presented. Finally, mean Li-Yorke chaos is also investigated for
$C_0$-semigroups of operators on Banach spaces.
\end{abstract}


\section{Introduction}

\shs In recent years, it has become popular in the area of dynamical systems
to investigate notions related to averages involving orbits or pseudo-orbits,
such as mean Li-Yorke chaos~\cite{GARJIN,HLY}, mean equicontinuity and mean
sensitivity~\cite{LTY}, and notions of shadowing with average error in
tracing~\cite{WOC}.

\smallskip
Our goal in this work is to investigate the notion of mean Li-Yorke chaos
for operators on Banach spaces. It turns out that this notion is intimately
related to the notion of absolutely mean irregular vector. Moreover,
we also establish some results on absolutely Ces\`aro bounded operators.

\smallskip
Let us now present the relevant definitions for our work.

\begin{definition}
An operator $T$ on a Banach space $X$ is said to be \emph{mean Li-Yorke
chaotic} if there is an uncountable subset $S$ of $X$ (a \emph{mean Li-Yorke
set} for $T$) such that every pair $(x,y)$ of distinct
points in $S$ is a \emph{mean Li-Yorke pair} for $T$, in the sense that
$$
\liminf_{N \to \infty} \frac{1}{N} \sum_{j=1}^N \|T^jx - T^jy\| = 0
\ \ \ \ \mbox{  and } \ \ \ \
\limsup_{N \to \infty} \frac{1}{N} \sum_{j=1}^N \|T^jx - T^jy\| > 0.
$$
If $S$ can be chosen to be dense (resp.\ residual) in $X$, then we say that
$T$ is \emph{densely} (resp.\ \emph{generically}) \emph{mean Li-Yorke chaotic}.
\end{definition}

\begin{definition}
Given an operator $T$ and a vector $x$, we say that $x$ is an
\emph{absolutely mean irregular} (resp.\ \emph{absolutely mean semi-irregular}) \emph{vector}
for $T$ if
$$
\liminf_{N \to \infty} \frac{1}{N} \sum_{j=1}^N \|T^j x\| = 0
\ \ \ \ \mbox{  and } \ \ \ \
\limsup_{N \to \infty} \frac{1}{N} \sum_{j=1}^N \|T^j x\| = \infty
\ (\mbox{ resp.} > 0).
$$
\end{definition}

\begin{definition} \cite{HouLuo15}
An operator $T$ is said to be \emph{absolutely Ces\`aro bounded} if there
exists a constant $C > 0$ such that
$$
\sup_{N \in \N} \frac{1}{N} \sum_{j=1}^N \|T^j x\| \leq C \|x\|
\ \ \mbox{  for all } x \in X.
$$
\end{definition}

Also related to our work is the notion of distributional chaos. Actually,
there are at least four different notions of distributional chaos, namely
DC1, DC2, DC$2\frac{1}{2}$ and DC3 (see Section~2 for the definitions).
For (nonlinear) dynamical systems on compact metric spaces, mean Li-Yorke
chaos is equivalent to DC2 \cite{D}, which is not equivalent to DC1
\cite{SS}. The situation is different in our context. Indeed, we saw in
\cite{BBPW} that DC1 and DC2 are always equivalent for operators on Banach
spaces, but we will show that DC2 is not equivalent to mean Li-Yorke chaos
in this context. The paper is organized as follows.

\smallskip
Section~2 recalls some definitions and fixes the notation.

\smallskip
In Section~3 we prove that an operator is mean Li-Yorke chaotic if and
only if it has an absolutely mean irregular vector. As a consequence,
no absolutely Ces\`aro bounded operator is mean Li-Yorke chaotic.
We also establish a Mean Li-Yorke Chaos Criterion and present several
examples, including an example of a DC1 (= DC2) operator which is not
mean Li-Yorke chaotic. Finally, we show that the Frequent Hypercyclicity
Criterion implies mean Li-Yorke chaos.

\smallskip
Section~4 is devoted to show that dense mean Li-Yorke chaos is equivalent to the
existence of a dense (or residual) set of absolutely mean irregular vectors.
As a consequence, every mean Li-Yorke chaotic operator is densely mean
Li-Yorke chaotic on some infinite-dimensional closed invariant subspace.
We also establish a Dense Mean Li-Yorke Criterion and several sufficient
conditions for dense mean Li-Yorke chaos. As an application, we give an
example of a densely mean Li-Yorke chaotic operator which is not Ces\`aro
Hypercyclic.

\smallskip
In Section~5 we establish a sufficient condition for the existence of a
dense absolutely mean irregular manifold. As an application, we obtain a
dichotomy for unilateral weighted backward shifts $B_w$, which says that
either $\frac{1}{N} \sum_{j=1}^N \|(B_w)^jx\| \to 0$ for all $x \in X$ or $B_w$
admits a dense absolutely mean irregular manifold. By using this dichotomy,
we give an example of a densely mean Li-Yorke chaotic operator which is not
hypercyclic.

\smallskip
In Section~6 we give some characterizations for generic mean Li-Yorke chaos
and construct an example of an invertible hypercyclic operator $T$ such that
both $T$ and $T^{-1}$ are completely absolutely mean irregular.

\smallskip
Finally, Section~7 is devoted to the study of mean Li-Yorke chaos for
$C_0$-semigroups of operators on Banach spaces.


\section{Notations and preliminaries}

\shs Throughout this paper $X$ will denote an arbitrary Banach space,
unless otherwise specified, and $L(X)$ is the space of all
bounded linear operators on $X$.

\smallskip
Let us recall that $T \in L(X)$ is said to be \emph{Li-Yorke chaotic} if
there exists an uncountable set $\Gamma \subset X$ such that for every pair
$(x,y) \in \Gamma \times \Gamma$ of distinct points, we have
\[
 \liminf_{n \to \infty} \|T^n x - T^n y\| = 0
 \ \ \mbox{ and } \ \
 \limsup_{n \to \infty} \|T^n x - T^n y\| > 0.
\]
$\Gamma$ is called a \emph{scrambled set} for $T$ and each such pair
$(x,y)$ is called a \emph{Li-Yorke pair} for $T$.

\smallskip
We also recall that the \emph{lower} and the \emph{upper densities} of a set
$A \subset \N$ are defined as
$$
\ldens(A):= \liminf_{n \to \infty} \frac{\card(A \cap [1,n])}{n}
 \ \ \mbox{  and } \ \
\udens(A):= \limsup_{n \to \infty} \frac{\card(A \cap [1,n])}{n},
$$
respectively. Given $T \in L(X)$, $x ,y \in X$ and $\delta > 0$, the
\emph{lower} and the \emph{upper distributional functions} of $x,y$
associated to $T$ are defined by
$$
F_{x,y}(\delta):= \ldens(\{j \in \N : \|T^j x - T^j y\| <
\delta\}),
$$
$$
F_{x,y}^*(\delta):= \udens(\{j \in \N : \|T^j x - T^j y\| <
\delta\}),
$$
respectively. If the pair $(x,y)$ satisfies
\begin{description}
\item {(DC1)} $F^*_{x,y} \equiv 1$ and $F_{x,y}(\eps) = 0$
              for some $\eps > 0$, or
\item {(DC2)} $F^*_{x,y} \equiv 1$ and $F_{x,y}(\eps) < 1$
              for some $\eps > 0$, or
\item {(DC2$\frac{1}{2}$)} There exist $c > 0$ and $r > 0$ such that
              $F_{x,y}(\delta) < c < F^*_{x,y}(\delta)$
              for all $ 0 < \delta < r$, or
\item {(DC3)} $F_{x,y}(\delta) < F^*_{x,y}(\delta)$ for all $\delta$
              in a nondegenerate interval $J \subset (0,\infty)$,
\end{description}
then $(x,y)$ is called a \emph{distributionally chaotic pair of type}
$k \in \{1,2,2\frac{1}{2},3\}$ for $T$. The operator $T$ is said to
be \emph{distributionally chaotic of type} $k$ (DC$k$) if there
exists an uncountable set $\Gamma \subset X$ such that every pair
$(x,y)$ of distinct points in $\Gamma$ is a distributionally chaotic
pair of type $k$ for $T$. In this case, $\Gamma$ is a \emph{ 
distributionally scrambled set of type $k$} for $T$. Distributional
chaos of type 1 is often called simply \emph{distributional chaos}.

\smallskip
For operators on Banach spaces, DC1 and DC2 are always equivalent
\cite[Theorem~2]{BBPW}, and imply Li-Yorke chaos. Li-Yorke chaos and
distributional chaos for linear operators have been studied in
\cite{BR,BBMP11,BB16,BBMP,BBMP2,BBPW,BPR17,HCC09,HouLuo15,MOP09,MOP13,WWC,YY16,YY17,YY18}, for instance.

\smallskip
Let us also recall that $T \in L(X)$ is \emph{frequently hypercyclic} (FH),
\emph{upper-frequently hypercyclic} (UFH), \emph{reiteratively hypercyclic}
(RH) or \emph{hypercyclic} (H) if there exists $x \in X$ such that for every
nonempty open subset $U$ of $X$, the set $\{n \in \N : T^n x \in U\}$
has positive lower density, has positive upper density, has positive
upper Banach density or is nonempty, respectively.
$T$ is \emph{Ces\`aro hypercyclic} if there exists $x \in X$ such that
the sequence $\big(\frac{1}{n} \sum_{j=0}^{n-1} T^jx\big)_{n \in \N}$
is dense in $X$. Moreover, $T$ is
\emph{mixing} if for every nonempty open sets $U,V \subset X$, there
exists $n_0 \in \N$ such that $T^n(U) \cap V \neq \emptyset$ for all
$n \geq n_0$, $T$ is \emph{weakly-mixing} if $T \oplus T$ is hypercyclic,
and $T$ is \emph{Devaney chaotic} if it is hypercyclic and has a dense set
of periodic points. See \cite{BaGr06,BM,BMPP,BoGE07,GM,GEP11}, for instance.

\smallskip
Finally, the orbit of $x$ is \emph{distributionally near to $0$} if there is
$A \subset \N$ with $\udens(A) = 1$ such that $\lim_{n \in A} T^nx = 0$.
We say that $x$ has a \emph{distributionally unbounded orbit} if there is
$B \subset \N$ with $\udens(B) = 1$ such that $\lim_{n \in B} \|T^nx\|
= \infty$. If the orbit of $x$ has both properties, then $x$ is a
\emph{distributionally irregular vector} for $T$. It was proved in
\cite{BBMP} that $T$ is distributionally chaotic if and only if $T$ has
a distributionally irregular vector.


\section{Mean Li-Yorke chaotic operators and absolutely mean irregular
vectors}

\shs We begin with some useful characterizations of absolutely Ces\`aro
bounded operators.

\begin{theorem}\label{ACB}
For every $T \in L(X)$, the following assertions are equivalent:
\begin{itemize}
\item [\rm (i)]   $T$ is not absolutely Ces\`aro bounded;
\item [\rm (ii)]  There is a vector $x \in X$ such that
      $$
      \sup_{N \in \N} \frac{1}{N} \sum_{j=1}^N \|T^j x\| = \infty;
      $$
\item [\rm (iii)] The set of all vectors $y \in X$ such that
      $$
      \sup_{N \in \N} \frac{1}{N} \sum_{j=1}^N \|T^j y\| = \infty
      $$
      is residual in $X$.
\end{itemize}
\end{theorem}

\begin{proof}
The implications (iii) $\Rightarrow$ (ii) $\Rightarrow$ (i) are trivial.

\smallskip
\noindent  (i) $\Rightarrow$ (ii): Since $T$ is not absolutely Ces\`aro
bounded, given $\delta > 0$ and $C > 0$, there exist $x \in X$ and
$N \in \N$ so that
$$
\|x\| < \delta
\ \ \ \ \mbox{  and } \ \ \ \
\frac{1}{N} \sum_{j=1}^N \|T^j x\| > C.
$$
Let us assume that (ii) is false, that is,
$$
\sup_{N \in \N} \frac{1}{N} \sum_{j=1}^N \|T^j x\| < \infty
\ \ \mbox{  for all } x \in X.
$$
Then, we can define inductively sequences $(x_n)$ in $X$ and $(N_n)$ in $\N$
so that $\|x_n\| < 2^{-n}$ for all $n$, and
$$
\frac{1}{N_k} \sum_{j=1}^{N_k} \|T^j(x_1+\cdots+x_n)\| > k
\ \ \mbox{  whenever } 1 \leq k \leq n.
$$
Let $x:= \sum_{n=1}^\infty x_n \in X$. Then
$$
\frac{1}{N_k} \sum_{j=1}^{N_k} \|T^j x\| \geq k \ \ \mbox{  for all } k \in \N.
$$
This is a contradiction, since we are assuming that (ii) is false.

\smallskip
\noindent (ii) $\Rightarrow$ (iii): Let $A$ denote the set considered
in (iii). Since
$$
A = \bigcap_{k=1}^\infty \bigcup_{N=1}^\infty
\Big\{y \in X : \frac{1}{N} \sum_{j=1}^N \|T^j y\| > k\Big\},
$$
we have that $A$ is a $G_\delta$ set. If $z \in X \backslash A$ then
$$
\sup_{N \in \N} \frac{1}{N} \sum_{j=1}^N \|T^j z\| < \infty.
$$
Let $x \in X$ be a vector given by (ii). Then
$$
\sup_{N \in \N} \frac{1}{N} \sum_{j=1}^N \|T^j(z + \lambda x)\| = \infty
\ \ \mbox{  whenever } \lambda \neq 0,
$$
which implies that $z \in \ov{A}$. Thus, $A$ is dense in $X$.
\end{proof}

Let us now prove that mean Li-Yorke chaos is equivalent to the existence
of an absolutely mean irregular vector.

\begin{theorem}\label{MLY}
For every $T \in L(X)$, the following assertions are equivalent:
\begin{itemize}
\item [\rm (i)]   $T$ is mean Li-Yorke chaotic;
\item [\rm (ii)]  $T$ has a mean Li-Yorke pair;
\item [\rm (iii)] $T$ has an absolutely mean semi-irregular vector;
\item [\rm (iv)]  $T$ has an absolutely mean irregular vector;
\item [\rm (v)]   The restriction of $T$ to some infinite-dimensional
      closed $T$-invariant subspace $Y$ has a residual set of
      absolutely mean irregular vectors.
\end{itemize}
\end{theorem}

\begin{proof}
The implications (v) $\Rightarrow$ (iv) $\Rightarrow$ (iii) are trivial.

\smallskip
\noindent (iii) $\Rightarrow$ (ii): If $x$ is an absolutely mean
semi-irregular vector for $T$, it is clear that $(x,0)$ is a
mean Li-Yorke pair for $T$.

\smallskip
\noindent (ii) $\Rightarrow$ (i): If $(x,y)$ is a mean Li-Yorke pair
for $T$, then $u:= x-y$ is an absolutely mean semi-irregular vector for $T$.
Hence, it follows easily that $\{\lambda u : \lambda \in \K\}$
is an uncountable mean Li-Yorke set for~$T$.

\smallskip
\noindent (i) $\Rightarrow$ (v): Let $(x,y)$ be a mean Li-Yorke pair
for $T$ and put $u:= x-y$. Let
$$
Y:= \ov{\spa}(\Orb(u,T)),
$$
which is an infinite-dimensional closed $T$-invariant subspace of $X$.
Consider the operator $S \in L(Y)$ obtained by restricting $T$ to $Y$.
We claim that $S$ is not absolutely Ces\`aro bounded. Indeed, suppose that
this is not the case and let $C > 0$ be such that
$$
\sup_{N \in \N} \frac{1}{N} \sum_{j=1}^N \|S^j z\| \leq C \|z\|
\ \ \mbox{  for all } z \in Y.
$$
Since $(x,y)$ is a mean Li-Yorke pair for $T$, we have that
$$
\liminf_{N \to \infty} \frac{1}{N} \sum_{j=1}^N \|S^j u\| = 0
\ \ \ \ \mbox{  and } \ \ \ \
\limsup_{N \to \infty} \frac{1}{N} \sum_{j=1}^N \|S^j u\| > \eps,
$$
for some $\eps > 0$. Let $\delta > 0$ be so small that
$$
\limsup_{N \to \infty} \frac{1}{N} \sum_{j=1}^N \|S^j u\| > \delta + \eps.
$$
There are $N \in \N$ and $M > N$ such that
$$
\frac{1}{N} \sum_{j=1}^N \|S^j u\| < \delta
\ \ \ \ \mbox{  and } \ \ \ \
\frac{1}{M} \sum_{j=1}^M \|S^j u\| > \delta + \eps.
$$
Let $K \in \{1,\ldots,N\}$ be the largest integer such that
$\|S^K u\| < \delta$. Then
$$
\frac{1}{K} \sum_{j=1}^K \|S^j u\| < \delta.
$$
Since
\begin{align*}
\delta + \eps &< \frac{1}{M} \sum_{j=1}^M \|S^j u\|\\
              &= \frac{K}{M} \frac{1}{K} \sum_{j=1}^K \|S^j u\|
                + \frac{M-K}{M} \frac{1}{M-K} \sum_{j=K+1}^M \|S^j u\|\\
              &< \delta + \frac{1}{M-K} \sum_{j=K+1}^M \|S^j u\|,
\end{align*}
we have that
$$
\eps < \frac{1}{M-K} \sum_{j=K+1}^M \|S^j u\|
     = \frac{1}{M-K} \sum_{j=1}^{M-K} \|S^j S^K u\|
  \leq C \|S^Ku\|
     < C \delta.
$$
Hence, $C > \eps/\delta$. Since $\delta > 0$ can be chosen arbitrarily
close to $0$, we have a contradiction. This proves that $S$ is not
absolutely Ces\`aro bounded. Therefore, the set
$$
A:= \Big\{z \in Y : \sup_{N \in \N} \frac{1}{N} \sum_{j=1}^N \|S^j z\|
                    = \infty \Big\}
$$
is residual in $Y$ by Theorem~\ref{ACB}. On the other hand, let
$$
B:= \Big\{z \in Y : \inf_{N \in \N} \frac{1}{N} \sum_{j=1}^N \|S^j z\|
                    = 0 \Big\}.
$$
Since
$$
B = \bigcap_{k=1}^\infty \bigcup_{N=1}^\infty
\Big\{z \in Y : \frac{1}{N} \sum_{j=1}^N \|S^j z\| < \frac{1}{k} \Big\},
$$
we have that $B$ is a $G_\delta$ set. Since there is an increasing sequence
$(N_k)$ in $\N$ such that
$$
\lim_{k \to \infty} \frac{1}{N_k} \sum_{j=1}^{N_k} \|S^j u\| = 0,
$$
it follows that $\spa(\Orb(u,S))$ is contained in $B$. Hence, $B$ is
a residual set in $Y$. Our conclusion is that $A \cap B$ is a residual
set in $Y$, which is formed by absolutely mean irregular vectors for $S$
(hence for $T$).
\end{proof}

\begin{corollary}\label{NotACB}
No absolutely Ces\`aro bounded operator on a Banach space is mean Li-Yorke
chaotic.
\end{corollary}

The proof of Theorem~\ref{MLY} clearly implies the following result:

\begin{theorem}\label{Semi-Irregular}
For every $T \in L(X)$, the set of all absolutely mean irregular vectors
for $T$ is dense in the set of all absolutely mean semi-irregular vectors
for $T$.
\end{theorem}

\begin{definition}\label{MLYCCDef}
We say that $T \in L(X)$ satisfies the \emph{Mean Li-Yorke Chaos Criterion}
(MLYCC) if there exists a subset $X_0$ of $X$ with the following properties:
\begin{itemize}
\item [(a)] $\displaystyle \liminf_{N \to \infty} \frac{1}{N}
            \sum_{j=1}^{N} \|T^j x\| = 0$ for every $x \in X_0$;
\item [(b)] there are sequences $(y_k)$ in $\ov{\spa(X_0)}$ and $(N_k)$
            in $\N$ such that
            $$
            \frac{1}{N_k}\sum_{j=1}^{N_k}\|T^j y_k\| > k \|y_k\|
            \ \mbox{  for every } k \in \N.
            $$
\end{itemize}
\end{definition}

Let us now prove that this criterion characterizes mean Li-Yorke chaos.

\begin{theorem}\label{MLYCC}
An operator $T \in L(X)$ is mean Li-Yorke chaotic if and only if it
satisfies the MLYCC.
\end{theorem}

\begin{proof}
$(\Rightarrow)$: By Theorem~\ref{MLY}, $T$ has an absolutely mean irregular
vector $x$. So, it is enough to consider $X_0:= \{x\}$.

\smallskip
\noindent $(\Leftarrow)$: If $T$ has an absolutely mean semi-irregular
vector, then $T$ is mean Li-Yorke chaotic by Theorem~\ref{MLY}. So, let us
assume that this is not the case. Then, (a) implies that
$$
\lim_{N \to \infty} \frac{1}{N} \sum_{j=1}^{N} \|T^j x\| = 0
\ \mbox{  for every } x \in X_0.
$$
Let
$$
Y:= \ov{\Big\{x \in X : \lim_{N \to \infty} \frac{1}{N}
                        \sum_{j=1}^{N} \|T^j x\| = 0 \Big\}},
$$
which is a closed $T$-invariant subspace of $X$. Consider the operator
$S \in L(Y)$ obtained by restricting $T$ to $Y$. Since the sequence
$(y_k)$ lies in $Y$, (b) implies that $S$ is not absolutely Ces\`aro
bounded. Hence, by Theorem~\ref{ACB}, the set
$$
A:= \Big\{y \in Y : \sup_{N \in \N} \frac{1}{N} \sum_{j=1}^N \|S^j y\|
                    = \infty \Big\}
$$
is residual in $Y$. On the other hand, since the set
$$
B:= \Big\{y \in Y : \inf_{N \in \N} \frac{1}{N} \sum_{j=1}^N \|S^j y\|
                    = 0 \Big\}
$$
is obviously dense in $Y$, it is residual in $Y$. Thus, $A \cap B$ is
residual in $Y$, which proves the existence of an absolutely mean irregular
vector for $T$. This contradicts our assumption that there are no
absolutely mean semi-irregular vector for $T$.
\end{proof}

\begin{example}\label{Ex1}
There are mixing operators that are not mean Li-Yorke chaotic.

\smallskip
Indeed, in \cite[Example 23]{BBPW} it was given some examples of mixing
operators that are absolutely Ces\`aro bounded. One of them was taken from \cite{MOP13}, and appears in the PhD Thesis of  Beltr\'an Meneu with a proof provided by the third author (Theorem 3.7.3 in \cite{beltran14}). By Corollary~\ref{NotACB},
these operators are not mean Li-Yorke chaotic.
\end{example}

\begin{example}\label{Ex2}
There are Devaney chaotic operators that are not mean Li-Yorke chaotic.

\smallskip
Indeed, let $T$ be the Devaney chaotic operator constructed by Menet \cite{M}.
Suppose that $x \neq 0$ and
$$
\liminf_{N \to \infty} \frac{1}{N} \sum_{j=1}^N \|T^j x\| = 0.
$$
For every $\delta > 0$, since
$$
\frac{\card(\{1 \leq j \leq N : \|T^jx\| \geq \delta\})}{N}
  \leq \frac{1}{N} \sum_{j=1}^N \frac{\|T^jx\|}{\delta}
     = \frac{1}{\delta} \frac{1}{N} \sum_{j=1}^N \|T^jx\|,
$$
we have that
$$
\ldens(\{n \geq 0 : \|T^n x\| \geq \delta\}) = 0.
$$
This contradicts Claim 10 in \cite{M}. Hence, $T$ is not mean Li-Yorke
chaotic.
\end{example}

\begin{example}\label{Ex3}
There are frequently hypercyclic operators that are not mean Li-Yorke chaotic.

\smallskip
Indeed, Bayart and Ruzsa \cite[Section 6]{BR} proved that there exists a
frequently hypercyclic operator (thus this operator is DC$2\frac{1}{2}$ by
\cite[Theorem 13]{BBPW}) such that the orbit of no $x \ne 0$ is
distributionally near of 0. Hence, it is not DC1 and not mean Li-Yorke
chaotic \cite[Proposition 20(a)]{BBPW}.
\end{example}

\begin{example}\label{Ex4}
There are distributionally chaotic operators that are not mean Li-Yorke
chaotic.

\smallskip
We recall the following result that was obtained in \cite{BBPW}:
\textit{ If $X = c_0(\N)$ ($X = c_0(\Z)$) or $X = \ell^p(\N)$ ($X = \ell^p(\Z)$)
for some $1 \leq p \leq \infty$, then there exists an (invertible) operator
$T \in L(X)$ which is distributionally chaotic and satisfies}
$$
\lim_{N \to \infty} \frac{1}{N} \sum_{j=1}^N \|T^jx\| = \infty
\ \ \mbox{ for all } x \in X \backslash \{0\}.
$$

Since these operators do not have an absolutely mean irregular vector,
Theorem~\ref{MLY} guarantees that they are not mean Li-Yorke chaotic.
Since DC1 and DC2 are equivalent for operators on Banach spaces
\cite[Theorem~2]{BBPW}, this implies that the notions of mean Li-Yorke
chaos and DC2 are not equivalent for operators on Banach spaces, contrary
to what happens in the context of topological dynamics on compact metric
spaces~\cite{D}.
\end{example}

\begin{remark}
The above examples are Li-Yorke chaotic, but not mean Li-Yorke chaotic.
\end{remark}

We recall that an operator $T$ on a separable Banach space $X$ is said to
satisfy the {\em Frequent Hypercyclicity Criterion} (FHC) if there exist
a dense subset $X_0$ of $X$ and a map $S : X_0 \to X_0$ such that,
for any $x \in X_0$,
\begin{itemize}
\item $\sum_{n=0}^\infty T^nx$ converges unconditionally,
\item $\sum_{n=0}^\infty S^nx$ converges unconditionally,
\item $TSx = x$.
\end{itemize}

If $T$ satisfies this criterion, then $T$ is frequently hypercyclic,
Devaney chaotic, mixing and distributionally chaotic \cite{BBMP,BoGE07}.
Since we have just seen that there are examples of frequently hypercyclic
operators, Devaney chaotic operators, mixing operators
and distributionally chaotic operators that are not mean Li-Yorke chaotic,
it is natural to ask if the Frequent Hypercyclicity Criterion implies mean
Li-Yorke chaos. Let us now see that the answer is yes.

\begin{proposition}\label{FHC}
If $T \in L(X)$ satisfies the Frequent Hypercyclicity Criterion, then $T$
has a residual set of absolutely mean irregular vectors.
\end{proposition}

\begin{proof}
If $T$ satisfies the Frequent Hypercyclicity Criterion, then the set
$$
\Big\{x \in X : \lim_{n \to \infty} \|T^nx\| = 0\Big\}
$$
is dense in $X$ and $T$ is distributionally chaotic. Hence, we can
apply \cite[Theorem 27]{BBPW}.
\end{proof}

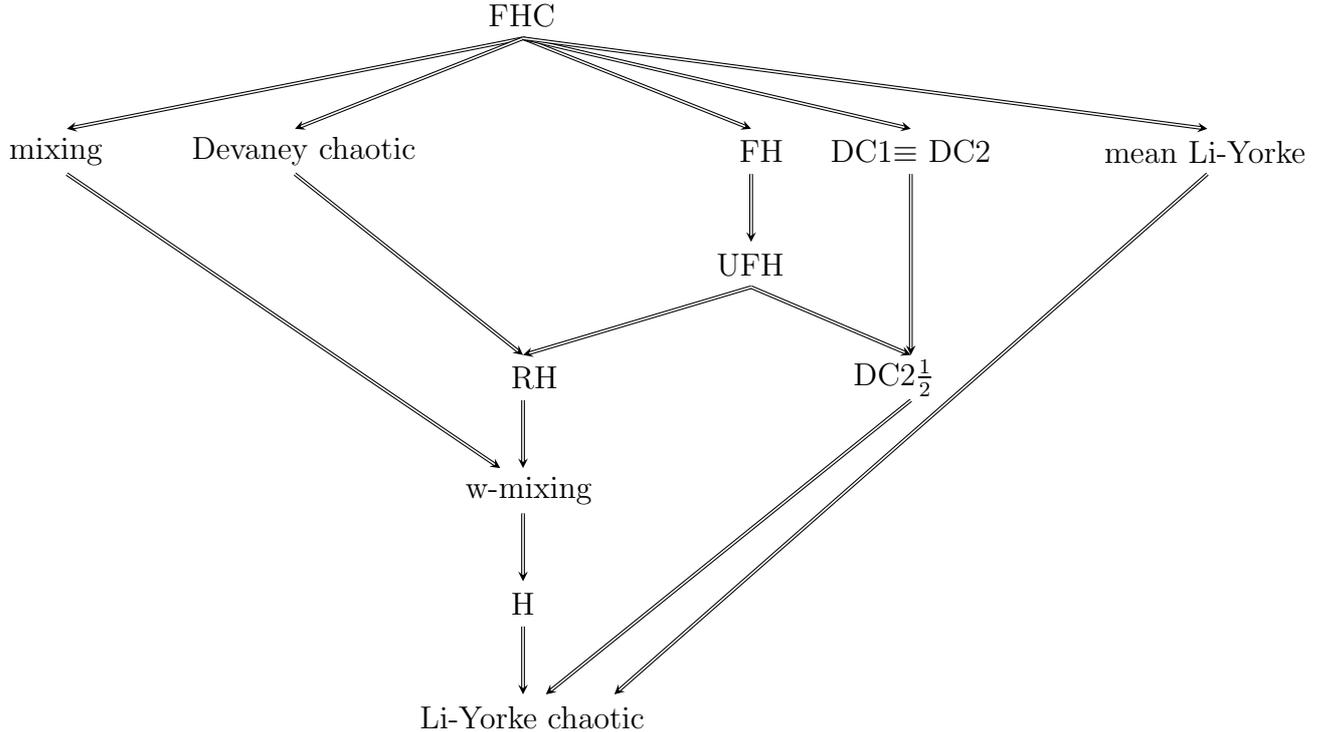
\begin{figure}[h]
\begin{tikzpicture}[scale=0.3,>=stealth]
 \node[right] at (18,21) {FHC};
  \node[right] at (-3,15) {mixing};
 \node[right] at (5,15) {Devaney chaotic};
  \node[right] at (29,15) {FH};
  \node[right] at (33,15) {DC1$\equiv$ DC2};
  \node[right] at (45,15) {mean Li-Yorke};
  \node[right] at (28,10) {UFH};
 \node[right] at (19,5) {RH};
 \node[right] at (34,5) {DC2$\frac{1}{2}$};
  \node[right] at (17,0) {w-mixing};
 \node[right] at (19,-5) {H};
  \node[right] at (15,-10) {Li-Yorke chaotic};
 \draw[double, ->] (20,20) -- (0,16);
  \draw[double, ->] (20,20) -- (10,16);
   \draw[double, ->] (20,20) -- (30,16);
  \draw[double, ->] (20,20) -- (37,16);
   \draw[double, ->] (20,20) -- (50,16);
      \draw[double, ->] (50,14) -- (24,-9);
   \draw[double, ->] (10,14) -- (20,6);
 \draw[double, ->] (0,14) -- (19,1);
  \draw[double, ->] (30,14) -- (30,11);
  \draw[double, ->] (30,9) -- (37,6);
  \draw[double, ->] (30,9) -- (20,6);
   \draw[double, ->] (37,14) -- (37,6);
\draw[double, ->] (20,4) -- (20,1); \draw[double, ->] (20,-1) --
(20,-4); \draw[double, ->] (37,4) -- (21,-9); \draw[double, ->]
(20,-6) -- (20,-9);
 \end{tikzpicture}
\caption{Implications between different definitions related with
hypercyclicity and chaos for operators on Banach spaces.}
\end{figure}

Concerning Figure~1, it is easy to construct a mean Li-Yorke
chaotic operator which is not hypercyclic. Indeed, let $T \in L(X)$ be
any mean Li-Yorke chaotic operator and consider the operator $T \oplus I$
on $X \oplus X$, where $I$ denotes the identity operator on $X$.
However, the following question remains open.

\begin{question}\label{Question1}
Is there a Banach (or Hilbert) space operator which is mean Li-Yorke chaotic
but is not distributionally chaotic?
\end{question}

Related to Question~\ref{Question1}, we have the following result from
\cite{BBPW}:
\textit{ If $X = c_0(\N)$ ($X = c_0(\Z)$) or $X = \ell^p(\N)$ ($X = \ell^p(\Z)$)
for some $1 \leq p \leq \infty$, then there exists an (invertible) operator
$T \in L(X)$ which admits an absolutely mean irregular vector whose orbit is
not distributionally unbounded.}

\smallskip
In particular, this shows that an absolutely mean irregular vector is not
necessarily distributionally irregular.  
However, we don't know if the operator admits other
vectors with distributionally unbounded orbit, or if it has a distributionally
irregular vector.


\section{Densely mean Li-Yorke chaotic operators}

\shs Let us now show that, in separable spaces, dense mean Li-Yorke chaos
is equivalent to the existence of a residual set of absolutely mean irregular
vectors.

\begin{theorem}\label{DMLY}
Assume $X$ separable. For every $T \in L(X)$, the following assertions are
equivalent:
\begin{itemize}
\item [\rm (i)]   $T$ is densely mean Li-Yorke chaotic;
\item [\rm (ii)]  $T$ has a dense set of mean Li-Yorke pairs;
\item [\rm (iii)] $T$ has a residual set of mean Li-Yorke pairs;
\item [\rm (iv)]  $T$ has a dense set of absolutely mean semi-irregular
                  vectors;
\item [\rm (v)]   $T$ has a dense set of absolutely mean irregular
                  vectors;
\item [\rm (vi)]  $T$ has a residual set of absolutely mean irregular
                  vectors.
\end{itemize}
\end{theorem}

\begin{proof}
(ii) $\Leftrightarrow$ (iv): It follows easily from the fact that $(x,y)$
is a mean Li-Yorke pair for $T$ if and only if $x-y$ is an absolutely mean
semi-irregular vector for $T$.

\smallskip
\noindent (iv) $\Leftrightarrow$ (v): It follows from
Theorem~\ref{Semi-Irregular}.

\smallskip
\noindent (v) $\Leftrightarrow$ (vi): The set $R$ of all absolutely mean
irregular vectors for $T$ is the intersection of the sets
$$
A:= \Big\{x \in X : \sup_{N \in \N} \frac{1}{N} \sum_{j=1}^N \|T^j x\|
                    = \infty \Big\}
\ \mbox{  and } \
B:= \Big\{x \in X : \inf_{N \in \N} \frac{1}{N} \sum_{j=1}^N \|T^j x\|
                    = 0 \Big\}.
$$
It follows from Theorem~\ref{ACB} that $A$ is residual in $X$ whenever it is
nonempty. And we saw in the proof of Theorem~\ref{MLY} that $B$ is residual
in $X$ whenever it is dense in $X$. Thus, $R$ is residual in $X$ whenever
it is dense in $X$.

\smallskip
\noindent (vi) $\Rightarrow$ (iii): If $(U_j)$ is a sequence of dense open
sets in $X$ such that every vector in $\bigcap U_j$ is absolutely mean
irregular for $T$, then the sets
$V_j := \{(x,y) \in X \times X : x - y \in U_j\}$
are open and dense in $X \times X$, and every point in $\bigcap V_j$ is a
mean Li-Yorke pair for $T$.

\smallskip
\noindent The implications (iii) $\Rightarrow$ (ii) and (i) $\Rightarrow$ (ii) are obvious.

\smallskip
\noindent (vi) $\Rightarrow$ (i): Let $R$ be the set of all absolutely
mean irregular vectors for $T$ and let $(y_j)$ be a dense sequence in $X$.
Let $D := \Q$ or $\Q + i\Q$, depending on whether the scalar field $\K$
is $\R$ or $\C$, respectively. Since $R$ is residual in $X$, we can choose
inductively linearly independent vectors $x_1,x_2,x_3,\ldots \in X$
such that $x_1 \in B(y_1;1) \cap R$ and
$$
x_{n+1} \in B\Big(y_{n+1};\frac{1}{n+1}\Big) \cap
  \bigcap_{(\alpha_1,\ldots,\alpha_n) \in D^n}
  (\alpha_1 x_1 + \cdots + \alpha_n x_n + R) \ \ \ (n \in \N).
$$
Hence,
$$
M := \{\alpha_1 x_1 + \cdots + \alpha_m x_m : m \geq 1 \mbox{  and }
       \alpha_1,\ldots,\alpha_m \in D\}
$$
is a dense $D$-vector subspace of $X$ consisting (up to $0$) of absolutely
mean irregular vectors for $T$. In particular, $M$ is a dense mean Li-Yorke
set for $T$. Since $M$ is countable, we need to enlarge $M$ in order to
obtain an uncountable dense mean Li-Yorke set for $T$. Let
$$
N := \{\alpha_2 x_2 + \cdots + \alpha_m x_m : m \geq 2 \mbox{  and }
       \alpha_2,\ldots,\alpha_m \in D\}.
$$
For each $y \in N \backslash \{0\}$, let
$
A_y:= \{\lambda \in \K : y - \lambda x_1 \mbox{  is absolutely mean irregular
        for } T\}.
$
Since $A_y$ is a $G_\delta$ set in $\K$ containing $D$, $A_y$ is residual
in $\K$. Thus,
$
A := \bigcap_{y \in N \backslash \{0\}} A_y
$
is also a residual set in $\K$ containing $D$. By Zorn's Lemma, there is
a maximal $D$-vector subspace $H$ of $\K$ such that $D \subset H \subset A$.
If $H$ were countable, then
$$
B:= \bigcap_{\beta \in D \backslash \{0\}} \bigcap_{\alpha \in H}
    \beta (\alpha + A)
$$
would be residual in $\K$. By choosing $\gamma \in B \backslash H$,
we would have that $H':= H + \{\beta \gamma : \beta \in D\}$ is a $D$-vector
subspace of $\K$ satisfying $D \subset H' \subset A$ and $H \subsetneq H'$,
which would contradict the maximality of $H$. Thus,
$M' := \{\alpha x_1 : \alpha \in H\} + N$
is the uncountable mean Li-Yorke set for $T$ we were looking for.
\end{proof}

\begin{remark}
Note that in the previous theorem the separability of $X$ was used only
in the proof that (vi) $\Rightarrow$ (i). The equivalences
$$
\mbox{ (ii)} \Leftrightarrow \mbox{ (iii)} \Leftrightarrow \mbox{ (iv)}
\Leftrightarrow \mbox{ (v)} \Leftrightarrow \mbox{ (vi)}
$$
are valid for any Banach space.
\end{remark}

In view of Theorem~\ref{MLY}, it is true that an operator $T \in L(X)$
is mean Li-Yorke chaotic if and only if there is a mean Li-Yorke set for $T$
which is a one-dimensional subspace of $X$. Nevertheless, we shall now show
that the chaotic behaviour always occurs in a much larger subspace of $X$.

\begin{theorem}\label{ID}
Every mean Li-Yorke chaotic operator $T \in L(X)$ is densely mean Li-Yorke
chaotic on some infinite-dimensional closed $T$-invariant subspace of $X$.
\end{theorem}

\begin{proof}
By Theorem~\ref{MLY}, the restriction of $T$ to a certain
infinite-dimensional closed $T$-invariant subspace $Y$ of $X$ has a
residual set of absolutely mean irregular vectors. The proof of
Theorem~\ref{MLY} actually constructs a separable such $Y$.
Hence, we can apply Theorem~\ref{DMLY} and conclude that $T$ is
densely mean Li-Yorke chaotic on $Y$.
\end{proof}

\begin{definition}\label{DMLYCCDef}
We say that $T \in L(X)$ satisfies the \emph{Dense Mean Li-Yorke Chaos
Criterion} (DMLYCC) if there exists a \emph{dense} subset $X_0$ of $X$
with properties~(a) and~(b) of~Definition~\ref{MLYCCDef}.
\end{definition}

\begin{theorem}\label{DMLYCC}
Assume $X$ separable. An operator $T \in L(X)$ is densely mean Li-Yorke
chaotic if and only if it satisfies the DMLYCC.
\end{theorem}

\begin{proof}
($\Rightarrow$): By Theorem~\ref{DMLY}, $T$ has a dense set $X_0$ of
absolutely mean irregular vectors. Clearly, $X_0$ satisfies properties~(a)
and~(b) of Definition~\ref{MLYCCDef}.

\smallskip
\noindent ($\Leftarrow$): Since $X_0$ is dense in $X$, condition (a)
implies that the set
$$
B:= \Big\{x \in X : \inf_{N \in \N} \frac{1}{N} \sum_{j=1}^N \|T^j x\|
                    = 0 \Big\}
$$
is residual in $X$. By (b), $T$ is not absolutely Ces\`aro bounded.
Hence, by Theorem~\ref{ACB}, the set
$$
A:= \Big\{x \in X : \sup_{N \in \N} \frac{1}{N} \sum_{j=1}^N \|T^j x\|
                    = \infty \Big\}
$$
is also residual in $X$. Thus, $T$ has a residual set of absolutely mean
irregular vectors. By Theorem~\ref{DMLY}, we conclude that $T$ is densely
mean Li-Yorke chaotic.
\end{proof}

\begin{theorem}\label{AdditionalCondition}
If $T \in L(X)$ and
$$
X_0:= \Big\{x \in X : \liminf_{N \to \infty} \frac{1}{N} \sum_{j=1}^N
                      \|T^j x\| = 0 \Big\}
$$
is dense in $X$, then the following assertions are equivalent:
\begin{itemize}
\item [\rm (i)]   $T$ is mean Li-Yorke chaotic;
\item [\rm (ii)]  $T$ has a residual set of absolutely mean irregular vectors;
\item [\rm (iii)] $T$ is not absolutely Ces\`aro bounded;
\item [\rm (iv)]  $\displaystyle \limsup_{N \to \infty} \frac{1}{N}
                  \sum_{j=1}^\infty \|T^j y_0\| > 0$ for some $y_0 \in X$.
\end{itemize}
\end{theorem}

\begin{proof}
(i) $\Rightarrow$ (iii): It follows from Corollary~\ref{NotACB}.

\smallskip
\noindent (iii) $\Rightarrow$ (ii): Since $X_0$ is dense in $X$ by
hypothesis, it is residual in $X$. Hence, (ii) follows from Theorem~\ref{ACB}.

\smallskip
\noindent The implication (ii) $\Rightarrow$ (i) follows from Theorem~\ref{MLY}.

\smallskip
\noindent The implication (i) $\Rightarrow$ (iv) is trivial.

\smallskip
\noindent (iv) $\Rightarrow$ (i): Suppose that $T$ is not mean Li-Yorke
chaotic. By Theorem~\ref{MLY}, there is no absolutely mean semi-irregular
vector for $T$, and so
$$
X_0 = \Big\{x \in X : \lim_{N \to \infty} \frac{1}{N} \sum_{j=1}^N
                      \|T^j x\| = 0 \Big\}.
$$
Thus, $X_0$ is a residual subspace of $X$, which implies that $X_0 = X$
and contradicts (iv).
\end{proof}

\begin{remark}
In the case that the space $X$ is separable, Theorem~\ref{DMLY} shows that
condition (ii) in the above theorem can be replaced by
\begin{itemize}
\item [\rm (ii')] $T$ is densely mean Li-Yorke chaotic.
\end{itemize}
\end{remark}

As an immediate consequence of the previous theorem, we have the following
dichotomy for unilateral weighted backward shifts on Banach sequence spaces.

\begin{corollary}
Let $X$ be a Banach sequence space in which $(e_n)_{n \in \N}$ is a basis
\cite[Section 4.1]{GEP11}. Suppose that the unilateral weighted backward
shift
$$
B_w(x_1,x_2,x_3,\ldots):= (w_2x_2,w_3x_3,w_4x_4,\ldots)
$$
is an operator on $X$. Then either
\begin{itemize}
\item [\rm (a)] $B_w$ is mean Li-Yorke chaotic, or
\item [\rm (b)] $B_w$ is absolutely Ces\`aro bounded.
\end{itemize}
\end{corollary}

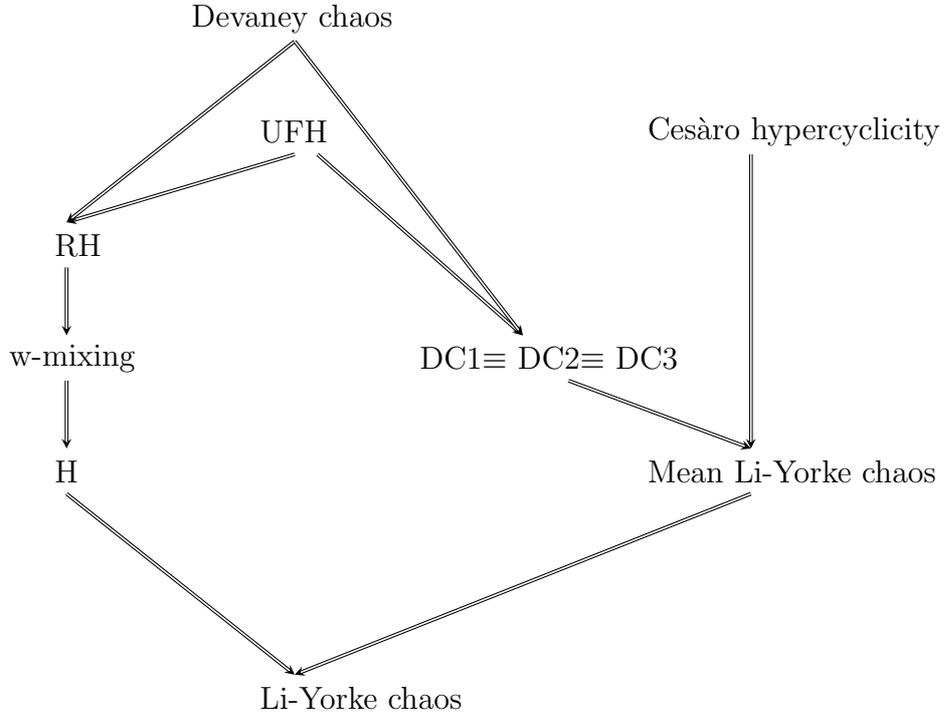
\begin{figure}[h]
\begin{tikzpicture}[scale=0.3,>=stealth]
  \node[right] at (10,20) {\textbf{ Remark:} If $\{ x: T^nx \rightarrow 0\}$ is a dense set in $X$, then: };
  \node[right] at (35,0) {DC1$\equiv $ DC2$\equiv $ DC3};
  \node[right] at (25,15) {Devaney chaos};
  \node[right] at (28,10) {UFH};
  \node[right] at (45,10) {Ces\`aro hypercyclicity};
  \node[right] at (45,-5) {Mean Li-Yorke chaos};
  \node[right] at (19,5) {RH};
  \node[right] at (17,0) {w-mixing};
 \node[right] at (19,-5) {H};
  \node[right] at (28,-15) {Li-Yorke chaos};
  \draw[double, ->] (31,9) -- (40,1);
  \draw[double, ->] (30,9) -- (20,6);
  \draw[double, ->] (50,9) -- (50,-4);
  \draw[double, ->] (42,-1) -- (50,-4);
   \draw[double, ->] (30,14) -- (20,6);
  \draw[double, ->] (20,4) -- (20,1);
  \draw[double, ->] (30,14) -- (40,1);
\draw[double, ->] (20,-1) -- (20,-4); \draw[double, ->] (20,-6)
--(30,-14);
\draw[double, ->] (50,-6) -- (30,-14);
 \end{tikzpicture}
\caption{Implications between different definitions related with
hypercyclicity and chaos for operators on Banach spaces when  $\{ x: T^nx
\rightarrow 0\}$ is a dense set in $X$.}
\end{figure}

Another consequence of Theorem~\ref{AdditionalCondition} is given in the following corollary.
Items (a) and (b) improve Theorems~27 and~28 of \cite{BBPW}, respectively.

\begin{corollary}\label{DC1}
Let $T \in L(X)$ be such that
$$
\Big\{x \in X : \liminf_{N \to \infty} \frac{1}{N} \sum_{j=1}^N \|T^j x\|
= 0\Big\}
$$
is dense in $X$. If any of the following conditions is true:
\begin{enumerate}
\item [(a)] $T$ is distributionally chaotic,
\item [(b)] $T$ is Ces\`aro hypercyclic,
\item [(c)] $X$ is a Banach space and $\displaystyle\limsup_{n \to \infty}
           \frac{\|T^n\|}{n} > 0$,
\item [(d)] $X$ is a Hilbert space and $\displaystyle \limsup_{n \to \infty}
           \frac{\|T^n\|}{n^{\frac{1}{2}}} > 0$,
\item [(e)] $T$ has an eigenvalue $\lambda$ with $|\lambda| \geq 1$,
\end{enumerate}
then there is a residual set of absolutely mean irregular vectors for $T$.
If, in addition, $X$ is separable, then $T$ is densely mean Li-Yorke chaotic.
\end{corollary}

\begin{proof}
Any of these conditions implies that
$$
\limsup_{N \to \infty} \frac{1}{N} \sum_{j=1}^N \|T^j y_0\| > 0
\ \mbox{  for some } y_0 \in X
$$
(for items (c) and (d) this follows from Theorem~2.4 and Corollary~2.6
of \cite{BermBMP}). Thus, the result follows from
Theorem~\ref{AdditionalCondition}.
\end{proof}

\begin{example}
There are densely mean Li-Yorke chaotic operators that are not Ces\`aro
hypercyclic.

\smallskip
Indeed, let $T$ be the weighted backward shift on $\ell^1(\N)$ defined by
$$
Te_1 = 0 \ \ \ \mbox{  and } \ \ \
Te_k = \Big(\frac{k}{k-1}\Big)e_{k-1} \ \mbox{  for } k > 1.
$$
Since $\|T^n\| = n + 1$ for all $n \in \N$, Corollary~\ref{DC1}(c) implies
that $T$ is densely mean Li-Yorke chaotic. Moreover, the equalities
$\|T^n\| = n + 1$ ($n \in \N$) also imply that $(\frac{T^nx}{n})$ is not
dense in $\ell^1(\N)$, for every $x \in \ell^1(\N)$. Thus, $T$ is not
Ces\`aro hypercyclic \cite{Leon}.
\end{example}

\begin{question}
Is the operator defined in the above example distributionally chaotic?
\end{question}


\section{Dense lineability of absolutely mean irregular vectors}

\begin{definition}
An \emph{absolutely mean irregular manifold} for $T \in L(X)$ is a vector
subspace $Y$ of $X$ such that every nonzero vector in $Y$ is absolutely
mean irregular for $T$.
\end{definition}

Such a manifold is clearly a mean Li-Yorke set for $T$. The following
dichotomy gives us a sufficient condition for the existence of a dense
absolutely mean irregular manifold.

\begin{theorem}\label{DAMIM}
Assume $X$ separable. If $T \in L(X)$ and
$$
X_0:= \Big\{x \in X : \lim _{N \to \infty} \frac{1}{N} \sum_{j=1}^N
            \|T^j x\| = 0\Big\}
$$
is dense in $X$, then either
\begin{itemize}
\item [\rm (a)] $\displaystyle \lim _{N \to \infty} \frac{1}{N} \sum_{j=1}^N
                \|T^j x\| = 0$ for every $x \in X$, or
\item [\rm (b)] $T$ admits a dense absolutely mean irregular manifold.
\end{itemize}
\end{theorem}

\begin{proof}
Suppose that (a) is false and let us prove (b).
By Theorem~\ref{AdditionalCondition}, $T$ has an absolutely mean irregular
vector. Let $C:= \|T\| > 1$. Then, we can construct a sequence $(x_m)$ of
normalized vectors in $X_0$ and an increasing sequence $(N_m)$ of positive
integers so that
$$
\frac{1}{N_m} \sum_{i=1}^{N_m} \|T^ix_m\| > m(2C)^m \ \mbox{  and } \
\frac{1}{N_m} \sum_{i=1}^{N_m} \|T^ix_k\| < \frac{1}{m} \ \mbox{  for }
k = 1,\ldots,m-1.
$$
Given $\alpha,\beta \in \{0,1\}^\N$, we say that $\beta\leq \alpha $
if $\beta_i\leq \alpha_i$ for all $i \in \N$. Let $(r_j)$ be a sequence
of positive integers such that $r_{j+1} \geq 1 + r_j + N_{r_j + 1}$ for all
$j \in \N$. Let $\alpha \in \{0,1\}^\N$ be defined by $\alpha_n = 1$ if and
only if $n = r_j$ for some $j \in \N$. For each $\beta \in \{0,1\}^\N$ such
that $\beta \leq \alpha$ and $\beta$ contains an infinite number of
$1$'s, we define
$$
x_\beta := \sum_i \frac{\beta_i}{(2C)^i}\, x_i
         = \sum_j \frac{\beta_{r_j}}{(2C)^{r_j}}\, x_{r_j}.
$$
Take $k \in \N$ with $\beta_{r_k} = 1$. Since
$\displaystyle\frac{1}{N_{r_k}}\sum_{i=1}^{N_{r_k}}\|T^i x_{r_k}\| >
r_k (2C)^{r_k}$ and
$\displaystyle \frac{1}{N_{r_k}}\sum_{i=1}^{N_{r_k}}\|T^i x_s\| <
\frac{1}{r_k}$ for each $s < r_k$, we have that
\begin{align*}
 \frac{1}{N_{r_k}}\sum_{i=1}^{N_{r_k}}\norm{T^i x_\beta}
 &\ge \frac{1}{(2C)^{r_k}} \frac{1}{N_{r_k}}\sum_{i=1}^{N_{r_k}}\norm{T^i x_{r_k}}
      - \frac{1}{N_{r_k}}\sum_{i=1}^{N_{r_k}}\sum_{j \neq k} \frac{\beta_{r_j}}{(2C)^{r_j}} \norm{T^i x_{r_j}} \\
 &>   r_k - \frac{1}{r_k} \sum_{j < k} \frac{1}{(2C)^{r_j}}
      - \sum_{j > k} \frac{\norm{x_{r_j}}}{2^{r_j}} \\
 &\ge r_k - 1.
\end{align*}
On the other hand,  since
$\displaystyle \frac{1}{N_{r_k+1}}\sum_{i=1}^{N_{r_k+1}}\norm{T^i x_s} <
\frac{1}{r_k + 1}$ for each $s < r_k + 1$, then
\begin{align*}
 \frac{1}{N_{r_k+1}}\sum_{i=1}^{N_{r_k+1}}\norm{T^i x_\beta}
 &\le  \frac{1}{N_{r_k+1}}\sum_{i=1}^{N_{r_k+1}}\sum_{j \le k} \frac{\beta_{r_j}\|T^i x_{r_j}\|}{(2C)^{r_j}}
    +  \frac{1}{N_{r_k+1}}\sum_{i=1}^{N_{r_k+1}}\sum_{j > k} \frac{\beta_{r_j}\|T^i x_{r_j}\|}{(2C)^{r_j}}\\
 &\le \frac{1}{r_k + 1} \sum_{j \le k} \frac{1}{(2C)^{r_j}}
    +  \frac{1}{N_{r_k+1}}\sum_{i=1}^{N_{r_k+1}}\sum_{j > k} \frac{\|x_{r_j}\|}{2^{r_j}}\\
 &<   \frac{1}{r_k + 1}\cdot
\end{align*}
Thus, $x_\beta$ is an absolutely mean irregular for $T$.

Now, let $(w_n)$ be a dense sequence in $X_0$ and choose
$\gamma_n \in \{0,1\}^\N$ ($n \in \N$) such that each $\gamma_n$
contains an infinite number of $1$'s, $\gamma_n \leq \alpha$ for every
$n \in \N$, and the sequences $\gamma_n$ have mutually disjoint supports.
Define $v_n := \sum_i \frac{\gamma_{n,i}}{(2C)^i}\, x_i$ and
$y_n:= w_n + \frac{1}{n}\, v_n$ ($n \in \N$).
Then $Y:= \spa\{y_n : n \in \N\}$ is a dense subspace of $X$.
Moreover, if $y \in Y \backslash \{0\}$, then we can write
$y = w_0 + \sum_k \frac{\rho_k}{(2C)^k}\, x_k$,
where $w_0 \in X_0$ and the sequence of scalars $(\rho_k)$ takes
only a finite number of values (each of them infinitely many times).
As in the above proof we can show that the vector
$v := \sum_k \frac{\rho_k}{(2C)^k}\, x_k$
is absolutely mean irregular for $T$. Since $y = w_0 + v$ and $w_0 \in X_0$,
we conclude that $y$ is also absolutely mean irregular for $T$.
\end{proof}

Here is an application of the previous theorem.

\begin{corollary}\label{Shift}
If $B_w$ is a unilateral weighted backward shift on a Banach sequence space
$X$ in which $(e_n)_{n \in \N}$ is a basis, then either
\begin{itemize}
\item [\rm (a)] $\displaystyle \lim _{N \to \infty} \frac{1}{N} \sum_{j=1}^N
                \|(B_w)^j x\| = 0$ for every $x \in X$, or
\item [\rm (b)] $B_w$ admits a dense absolutely mean irregular manifold.
\end{itemize}
\end{corollary}

\begin{example}
There are densely mean Li-Yorke chaotic operators that are not hypercyclic.

\smallskip
Indeed, let $X = c_0(\N)$ or $X = \ell^p(\N)$ for some $1 \leq p < \infty$.
Consider the unilateral weighted backward shift $B_w : X \to X$ whose
weight sequence is given by
$$
w:= \Big(\frac{1}{2},2,\frac{1}{2},\frac{1}{2},2,2,\frac{1}{2},\frac{1}{2},
         \frac{1}{2},2,2,2,\ldots\Big),
$$
with successive blocks of $\frac{1}{2}$'s and 2's. Since
$\sup_{n \in \N} \prod_{j=1}^n |w_j| = 1 < \infty$,  
$B_w$ is not hypercyclic \cite[Example~4.9]{GEP11}. On the other hand,
if we define $x \in X$ by putting $\frac{1}{2^n}$ in the position of the
last $2$ in the $n^{th}$ block of $2$'s, for each $n \in \N$, and
$0$ otherwise, then
$$
\|(B_w)^nx\| \geq 1 \ \mbox{  for all } n \in \N.
$$
Hence, Corollary~\ref{Shift} guarantees that $B_w$ is densely mean Li-Yorke
chaotic.
\end{example}

\begin{remark}\label{R}
In Corollary \ref{DC1}, if $X$ is separable and we assume the stronger
property that
$$
\Big\{x \in X : \lim _{N \to \infty} \frac{1}{N} \sum_{j=1}^N
      \|T^j x\| = 0\Big\}
$$
is dense in $X$, then we can conclude that $T$ has a dense absolutely
mean irregular manifold.
\end{remark}


\section{Generically mean Li-Yorke chaotic operators}

\shs We have the following characterizations of generic mean Li-Yorke chaos.

\begin{theorem}\label{GMLY}
For every $T \in L(X)$, the following assertions are equivalent:
\begin{itemize}
\item [\rm (i)]   $T$ is generically mean Li-Yorke chaotic;
\item [\rm (ii)]  Every non-zero vector is absolutely mean semi-irregular
                  for $T$;
\item [\rm (iii)] $X$ is a mean Li-Yorke set for $T$.
\end{itemize}
\end{theorem}

The proof is analogous to that of \cite[Theorem 34]{BBMP2}.

\begin{definition}
We say that an operator $T \in L(X)$ is \emph{completely absolutely mean
irregular} if every vector $x \in X \backslash \{0\}$ is absolutely mean
irregular for $T$.
\end{definition}

Thus, every completely absolutely mean irregular operator is generically
mean Li-Yorke chaotic. The converse is not true in general
(see Remark~\ref{NotCAMI}).

\smallskip
Our next goal is to construct an invertible hypercyclic operator $T$
such that both $T$ and $T^{-1}$ are completely absolutely mean irregular.
The construction is a modification of the type of examples of completely
distributionally irregular operators provided in \cite{MOP13}.
We first recall one of the main results in \cite{MOP13}.

\begin{theorem}\cite[Thm. 3.1]{MOP13}\label{cdc_f_bil}
Let $v=(v_j)_{j\in \Z}$ be a weight sequence that satisfies the following conditions:
\begin{enumerate}
\item there are sequences of integers $(n_j)_{j\in \Z}$ and $(m_j)_{j\in \Z}$ with
$n_j<m_j<n_{j+1}$, $j\in \Z$, and $M>1$ such that $Mv_{m_{-k}}\geq v_j$ for every $j\in[m_{-k},m_{k-1}]$, $k\in\N$, and if we consider
\[
S_k:=\sup \Big\{ \frac{v_j}{v_{j-1}} \ ; \ j\not\in\; ]m_{-k},m_{k-1}]\Big\} , \ \ k\in\N ,
\]
then for every $\varepsilon>0$  we find $k\in \N$ with $v_{n_k}<\varepsilon$ and
\[
S_k^{k(n_k-m_{-k})}\leq \min \left\{ M,\frac{\min\{v_i; \ m_{-k}\leq i\leq m_{k-1}\}}{v_{n_k}}\right\},
\]
\item for every $N\in \N$, there exists $k\in\N$ such that $v_j>N$, for  $k\leq j\leq Nk$.
\end{enumerate}
Then the forward shift $T\colon \ell^p(v,\Z) \to \ell^p(v,\Z)$ is completely distributionally irregular.
\end{theorem}

Actually, it was shown that, given an arbitrary non-zero vector
$x \in \ell^p(v,\Z)$, and an arbitrary $\delta > 0$, there is $k \in \N$
as big as we want such that
\[
(I1) \ \ \ \ \norm{B^lx}^p<\delta \ \mbox{ for any } \ l\in [n_k-m_{-k},k(n_k-m_{-k})].
\]

The type of examples that we will consider involve the inverse too,
and we also need to recall the following result:

\begin{corollary}\cite[Cor.\ 3.3]{MOP13}\label{cdc_b_bil2}
Let $v=(v_j)_{j\in \Z}$ be a weight sequence that satisfies the following conditions:
\begin{enumerate}
\item
there are sequences of integers $(n_j)_{j\in \Z}$ and $(m_j)_{j\in \Z}$ with
$n_j<m_j<n_{j+1}$, $j\in \Z$, and  $M>1$ such that $Mv_{m_{k}}\geq v_j$ for every $j\in[m_{-k},m_{k}]$, $k\in\N$, and if we consider
\[
s_k:=\inf \Big\{ \frac{v_j}{v_{j-1}} \ ; \ j\not\in\; ]m_{-k},m_{k}]\Big\}, \ k\in\N,
\]
then for every $\varepsilon>0$  we find $k\in \N$ with $v_{n_{-k}}<\varepsilon$ and
\[
s_k^{k(n_{-k}-m_{k})}\leq \min \left\{ M,\frac{\min\{v_i; \ m_{-k}\leq i\leq m_{k}\}}{v_{n_{-k}}}\right\},
\]
\item for every $N\in \N$, there exits $k\in\N$ such that $v_j>N$, for  $-Nk\leq j\leq -k$.
\end{enumerate}
Then the backward shift $B=T^{-1}\colon\ell^p(v,\Z) \to \ell^p(v,\Z)$
is completely distributionally irre\-gular.
\end{corollary}

In this case, it turns out that, given an arbitrary non-zero vector
$x \in \ell^p(v,\Z)$, and an arbitrary $\delta > 0$, there is $k \in \N$
as big as we want such that
\[
(I2) \ \ \ \ \norm{B^lx}^p<\delta \ \mbox{ for any } \ l\in [m_k-n_{-k},k(m_k-n_{-k})].
\]

These results allow us to provide examples of completely absolutely mean
irregular operators, which are a modification of Example 3.5 of \cite{MOP13}.
The sequences of integers $(n_j)_{j\in \Z}$ and $(m_j)_{j\in \Z}$ with
$n_j<m_j<n_{j+1}$, $j\in \Z$, are such that for every $k\in \Z$ we have
\[
v_{j-1}\leq v_j \mbox{ when } \ n_k<j\leq m_k, \mbox{ and } v_{j-1}\geq v_j \mbox{ when } \
m_k<j\leq n_{k+1}.
\]
In other words, the positions $v_{m_k}$ represent ``hills'' of the weight sequence, and the positions
$v_{n_k}$ are ``valleys''.

\begin{theorem}\label{t_bil_cami}
There exists a bilateral shift $T$ on $\ell^p(v,\Z)$ such that both $T$
and $T^{-1}$ are completely absolutely mean irregular, completely
distributionally irregular and hypercyclic.
\end{theorem}

\begin{proof}
We first consider
\begin{enumerate}
\item [(a)] $n_0 = -1$, $m_0 = 1$, $n_1 = 4$, $m_{-1} = -4$, $v_{n_0} = 1$,
      $v_{m_0} = 2^{1/4}$, $v_{n_1} = 2^{-1/3}$, $v_{m_{-1}} = 2^{1/4}$, and
\item [(b)] $m_k = -n_{-k}$, $k \in \Z$, $v_{n_k} = (2k)^{-1/3}$,
      $v_{m_k} = (k+2)^{1/4}$, $v_{n_{-k}} = (2k+1)^{-1/3}$,
      $v_{m_{-k}} = (k+1)^{1/4}$, $k \in \N$, $v_i/v_{i-1} = v_j/v_{j-1}$
      if $i,j \in\; ]n_k,m_k]$, or if $i,j \in\; ]m_{k-1},n_k]$, $k\in \Z$, and
\item [(c)] $m_k-n_k > 2(m_{k-1}-n_{k-1})$, $n_{k+1}-m_k > 2(n_{k}-m_{k-1})$,
      $k \in \N$.
\end{enumerate}

We will check that the hypotheses of Theorem~\ref{cdc_f_bil} and
Corollary~\ref{cdc_b_bil2} are satisfied. Condition (b) gives
$\displaystyle \frac{\min\{v_i; \ m_{-k}\leq i\leq m_{k-1}\}}{v_{n_k}}=\frac{v_{n_{-k+1}}}{v_{n_k}}=\left(\frac{2k}{2k-1}\right)^{1/3}$ for
every $k\in\N$, and the  supremum of  the slope of $v$ outside the  interval $[m_{-k},m_{k-1}]$ is $S_k=v_j/v_{j-1}$
for any $n_k<j\leq m_k$, $k\in\N$. We set $M=2$ and, since $S_k^{m_k-n_k}=\frac{v_{m_k}}{v_{n_k}}=(2k)^{1/3}(k+2)^{1/4}$, we need that
\[
S_k^{k(n_k-m_{-k})}=((2k)^{2/3}(k+2)^{1/2})^{\frac{kn_k}{(m_k-n_k)}} \leq \frac{\min\{v_i; \ m_{-k}\leq i\leq m_{k-1}\}}{v_{n_k}}=\left(\frac{2k}{2k-1}\right)^{1/3} .
\]

This can be obtained for, e.g., $m_k = (16k^3+1)n_k$, $k \in \N$. Indeed,
for this selection and for any $k \geq 2$ (the case $k = 1$ trivially
satisfies the above inequality), we have
\[
S_k^{k(n_k-m_{-k})}<(2k(k+2))^{1/24k^2}\leq (2k)^{1/12k^2}\leq \left(\frac{2k}{2k-1}\right)^{1/3} .
\]
Thus, $T$ satisfies $(I1)$. If $l\leq n_k-m_{-k}$, then
\[
\norm{T^lx}\leq \frac{v_{m_k}}{v_{n_k}}\norm{x}\leq  (2k)^{2/3}\norm{x},
\]
which yields that, given an arbitrary non-zero vector $x \in \ell^p(v,\Z)$
and an arbitrary $\delta > 0$, there is $k \in \N$ as big as we want such
that
\[
\frac{1}{N}\sum_{j=1}^N \norm{T^lx}< \frac{(2k)^{2/3}}{k+1}  \norm{x} + \delta ,
\]
for $N=k(n_k-m_{-k})$, and we obtain that $T$ is completely absolutely mean
irregular as soon as we show that condition \emph{(2)} in Theorem~\ref{cdc_f_bil} is satisfied. For this, we notice that
\[
v_r=v_{m_k}s_k^{r-m_k} \geq (k+2)^{1/4}s_k^{2km_k}>(k+2)^{1/4}\left( \frac{2k}{2k+1}\right)^{1/3} \ \mbox{ if } \ m_k\leq r\leq (2k+1)m_k,
\]
which yields that $T$ is completely absolutely mean
irregular.

Analogously, $\displaystyle \frac{\min\{v_i; \ m_{-k}\leq i\leq m_{k}\}}{v_{n_{-k}}}=\frac{v_{n_{k}}}{v_{n_{-k}}}=\left(\frac{2k+1}{2k}\right)^{1/3}$ for
every $k\in\N$, and the  infimum of  the slope of $v$ outside the  interval $[m_{-k},m_{k}]$ is $s_k=v_j/v_{j-1}$
for any $m_k<j\leq n_{k+1}$, $k\in\N$. Again,   we set $M=2$ and, since $s_k^{m_k-n_{k+1}}=((2k+2)^{1/3}(k+2)^{1/4})$, we need that
\[
s_k^{k(n_{-k}-m_{k})}=s_k^{-2km_k}=((2k+2)^{2/3}(k+2)^{1/2})^{\frac{km_k}{(n_{k+1}-m_k)}} \leq \left(\frac{2k+1}{2k}\right)^{1/3} ,
\]
which is easily satisfied if we set, e.g., $n_{k+1}=(16k^3+1)m_k$, $k\in\N$. Thus, $B=T^{-1}$ satisfies $(I2)$. If $l\leq m_k-n_{-k}$, then
\[
\norm{B^lx}\leq \frac{v_{m_k}}{v_{n_k}}\norm{x}\leq  (2k)^{2/3}\norm{x}.
\]
As before, we obtain that $B=T^{-1}$ is completely absolutely mean irregular in case that condition \emph{(2)} in  Corollary~\ref{cdc_b_bil2} is satisfied. Indeed, we easily have that
\[
v_r>(k+1)^{1/4}\left( \frac{2k}{2k+1}\right)^{1/3} \ \mbox{ if } \ (2k+1)m_{-k}\leq r\leq m_{-k} ,
\]
which implies condition \emph{(2)}.

For the hypercyclicity of $T$, since it is invertible, it suffices to show
that there is an increasing sequence $(j_k)_k$ in $\N$ such that
$\lim_kv_{j_k} = \lim_kv_{-j_k} = 0$ (See Theorem 3.2 in \cite{F}).
Let $j_k:=(m_k+n_k)/2=(8k^3+1)n_k$. We have
\[
v_{j_k}=S_k^{j_k-n_k}v_{n_k}=\left( (2k)^{1/6}(k+2)^{1/8}\right)    \frac{1}{(2k)^{1/3}},
\]
that tends to $0$ as $k$ goes to infinity. Note that $R_k = v_j/v_{j-1}$
has the same value for any  $j \in\; ]n_{-k},m_{-k}]$, and thus for all
$k \in \N$ we have $R_k < S_k$ and
\[
v_{-j_k}=R_k^{-j_k-n_{-k}}v_{n_{-k}}<S_k^{m_k-j_k}\frac{1}{(2k+1)^{1/3}}\cdot
\]
Therefore, $\lim_kv_{j_k} = \lim_kv_{-j_k} = 0$, which concludes the
hypercyclicity of $T$.
\end{proof}

\begin{corollary}
There exists an operator $T$ on a Banach space $X$ such that the whole
space $X$ is a mean Li-Yorke set for $T$.
\end{corollary}

\begin{remark}
Since mean Li-Yorke chaos and DC2 are equivalent for dynamical systems
on compact metric spaces, it follows from \cite[Theorem~3.1]{F-KOS} that
the conclusion of the above corollary is not possible for such dynamical
systems.
\end{remark}

\begin{remark}\label{NotCAMI}
It is possible to modify slightly the example in Theorem~\ref{t_bil_cami}
in order to obtain $T$ such that every non-zero vector is absolutely mean
semi-irregular for both $T$ and $T^{-1}$ (thus, both operators are
generically mean Li-Yorke chaotic), but neither $T$ nor $T^{-1}$ are
completely absolutely mean irregular. To do this, the only change would be
to set $v_{m_k}=1$, $k\in\Z$. That is,
\begin{enumerate}
\item [(a)] $n_0 = -1$, $m_0 = 1$, $n_1 = 4$, $m_{-1} = -4$, $v_{n_0} = 1$,
      $v_{m_0} = 1$, $v_{n_1} = 2^{-1/3}$, $v_{m_{-1}} = 1$, and
\item [(b)] $m_k = -n_{-k}$, $k \in \Z$, $v_{n_k} = (2k)^{-1/3}$,
      $v_{m_k} = 1$, $v_{n_{-k}} = (2k+1)^{-1/3}$,
      $v_{m_{-k}} = 1$, $k \in \N$, $v_i/v_{i-1} = v_j/v_{j-1}$
      if $i,j \in\; ]n_k,m_k]$, or if $i,j \in\; ]m_{k-1},n_k]$, $k\in \Z$, and
\item [(c)] $m_k = (16k^3+1)n_k$, $n_{k+1} = (16k^3+1)m_k$,
      $k \in \N$.
\end{enumerate}
In that case, the vectors of the unit basis are not absolutely mean irregular.
Also, we observe that the hypercyclicity condition is preserved.
\end{remark}


\section{Mean Li-Yorke chaotic semigroups}

\shs We recall that a one-parameter family $(T_t)_{t \geq 0}$ of operators
on $X$ is called a \emph{$C_0$-semigroup} if $T_0 = I$, $T_t T_s = T_{t+s}$
($t,s \geq 0$) and $\lim_{t \to s} T_tx = T_sx$ ($x \in X$ and $s \geq 0$).
It is well-known that such a semigroup is always \emph{locally equicontinuous},
in the sense that
$$
\sup_{t \in [0,b]} \|T_t\| < \infty \ \mbox{  for every } b > 0.
$$
We refer the reader to the book \cite{EN} for a detailed study of
$C_0$-semigroups. In the sequel, $\cT = (T_t)_{t \geq 0}$ will denote
an arbitrary $C_0$-semigroup, unless otherwise specified.

\begin{definition}
$\cT$ is said to be \emph{mean Li-Yorke chaotic} if there is an uncountable
subset $S$ of $X$ (a \emph{mean Li-Yorke set} for $\cT$) such that every pair
$(x,y)$ of distinct points in $S$ is a \emph{mean Li-Yorke pair} for $\cT$,
in the sense that
$$
\liminf_{b \to \infty} \frac{1}{b} \int_0^b \|T_tx - T_ty\|dt = 0
\ \ \ \mbox{  and } \ \ \
\limsup_{b \to \infty} \frac{1}{b} \int_0^b \|T_tx - T_ty\|dt > 0.
$$
If $S$ can be chosen to be dense (resp.\ residual) in $X$, then we say that
$\cT$ is \emph{densely} (resp.\ \emph{generically}) \emph{mean Li-Yorke chaotic}.
\end{definition}

Li-Yorke chaos and distributional chaos for $C_0$-semigroups were studied in 
\cite{ABMP,BC12,BP12,CLMP,CRT,W14}, for instance. 

\begin{definition}
$\cT$ is called \emph{absolutely Ces\`aro bounded} if there is $C > 0$
such that
$$
\sup_{b > 0} \frac{1}{b} \int_0^b \|T_tx\|dt \leq C \|x\|
\ \ \mbox{  for all } x \in X.
$$
\end{definition}

\begin{definition}
We say that $x \in X$ is an \emph{absolutely mean irregular}
(resp.\ \emph{absolutely mean semi-irregular}) \emph{vector} for $\cT$ if
$$
\liminf_{b \to \infty} \frac{1}{b} \int_0^b \|T_tx\|dt = 0
\ \ \ \mbox{  and } \ \ \
\limsup_{b \to \infty} \frac{1}{b} \int_0^b \|T_tx\|dt = \infty
\ (\mbox{ resp.} > 0).
$$
\end{definition}

\begin{definition}
An \emph{absolutely mean irregular manifold} for $\cT$ is a vector subspace
$Y$ of $X$ such that every nonzero vector in $Y$ is absolutely mean irregular
for $\cT$.
\end{definition}

As we did in \cite{CMP} for hypercyclicity, or in \cite{ABMP} for
distributional chaos, we will establish equivalences between the above
notions for $C_0$-semigroups and the corresponding ones for the operators
of the semigroup. The following simple lemma (whose proof we omit) will be
very useful for this purpose.

\begin{lemma}
For each $s > 0$, let $C_s:= \sup_{t \in [0,s]} \|T_t\| < \infty$. Then
$$
\frac{1}{C_s} \frac{1}{N+1} \sum_{j=1}^N \|(T_s)^j x\| \leq
\frac{1}{b} \int_0^b \|T_tx\|dt \leq
C_s \frac{1}{N} \sum_{j=0}^N \|(T_s)^jx\|,
$$
whenever $x \in X$ and $Ns \leq b < (N+1)s$ with $N \geq 1$.
\end{lemma}

The next two propositions follow easily from this lemma.

\begin{proposition}
Given $x \in X$ and a $C_0$-semigroup $\cT$, the following are equivalent: 
\begin{enumerate} 
\item $x$ is an absolutely mean semi-irregular (resp.\ irregular)
      vector for $\cT$;
\item there exists $s>0$ such that $x$ is an absolutely mean semi-irregular (resp.\ irregular)
      vector for $T_s$; 
\item $x$ is an absolutely mean semi-irregular (resp.\ irregular)
      vector for $T_s$, for all $s > 0$. 
\end{enumerate} 
\end{proposition}

\begin{proposition}
Given a $C_0$-semigroup $\cT$, the following are equivalent: 
\begin{enumerate} 
\item $\cT$ is absolutely Ces\`aro bounded (resp.\ (densely, generically)
      mean Li-Yorke chaotic, admits a dense absolutely mean irregular
      manifold);
\item there exists $s>0$ such that $T_s$ is absolutely Ces\`aro bounded
      (resp.\ (densely, generically) mean Li-Yorke chaotic,
      admits a dense absolutely mean irregular manifold);
\item $T_s$ is absolutely Ces\`aro bounded (resp.\ (densely, generically)
      mean Li-Yorke chaotic, admits a dense absolutely mean irregular
      manifold), for all $s>0$.
\end{enumerate} 
\end{proposition}

\begin{remark}
With these propositions at hand, it is easy to transport many of our
previous theorems on operators to the semigroup setting. For instance,
Theorems~\ref{MLY} (without~(v)), \ref{DMLY} and \ref{GMLY} remain valid
if we replace the operator $T$ by the semigroup $\cT$. In particular,
absolutely Ces\`aro bounded $C_0$-semigroups are never mean Li-Yorke
chaotic.
\end{remark}

\begin{theorem}\label{MLYCS}
Suppose that the set
$\Big\{x \in X : \displaystyle \lim_{b \to \infty} \frac{1}{b}
\int_0^b \|T_t x\|dt = 0\Big\}$ is dense in $X$. If any of the following
conditions is true:
\begin{itemize}
\item [(a)] $\displaystyle \limsup_{b \to \infty} \frac{1}{b}
            \int_0^b \|T_t y_0\|dt > 0$ for some $y_0 \in X$,
\item [(b)] $X$ is a Banach space and $\displaystyle \limsup_{t \to \infty}
            \frac{\|T_t\|}{t} > 0$,
\item [(c)] $X$ is a Hilbert space and $\displaystyle \limsup_{t \to \infty}
            \frac{\|T_t\|}{t^{\frac{1}{2}}} > 0$,
\item [(d)] there is some $\lambda \in \sigma_p(A)$ with $\re \lambda \geq 0$,
            where $A$ is the infinitesimal generator of $\cT$,
\end{itemize}
then $\cT$ has a residual set of absolutely mean irregular vectors.
If, in addition, $X$ is separable, then $\cT$ admits a dense absolutely
mean irregular manifold.
\end{theorem}

\begin{proof}
It is a consequence of Theorem~\ref{AdditionalCondition} and
Corollary~\ref{DC1} (note that (b) and (c) imply that
$\displaystyle \limsup_{n \to \infty} \frac{\|T_1^n\|}{n} > 0$ and
$\displaystyle \limsup_{n \to \infty} \frac{\|T_1^n\|}{n^{\frac{1}{2}}} > 0$,
respectively). For the last assertion, see Remark~\ref{R}.
\end{proof}

As an immediate consequence of the previous theorem, we have the following
dichotomy for translation semigroups on weighted $L^p$ spaces.

\begin{corollary}
Let $v : \R_+ \to \R$ be an admissible weight function and consider
the translation semigroup $\cT$, given by
$$
T_t(f)(x) = f(x+t), \ \ t,x \geq 0,
$$
on the space $L^p_v(\R_+)$ \cite[Example~7.4]{GEP11}. Then either
\begin{itemize}
\item [\rm (a)] $\displaystyle \lim_{b \to \infty} \frac{1}{b}
                \int_0^b \|T_tf\|dt = 0$ for every $f \in L^p_v(\R_+)$, or
\item [\rm (b)] $\cT$ admits a dense absolutely mean irregular manifold.
\end{itemize}
\end{corollary}

As a consequence of Theorem~\ref{MLYCS}(b) and \cite[page 224]{MV}, we obtain

\begin{corollary}\label{L1}
The $C_0$-semigroup ${\cal T}$ defined on $L^1(1,\infty)$ by
$$
T_tf(x):= \Big(\frac{x+t}{x}\Big )f(x+t)
$$
is mean Li-Yorke chaotic.
\end{corollary}

\begin{question}
Is there a $C_0$-semigroup which is mean Li-Yorke chaotic but is not
distributionally chaotic?
Is the semigroup defined in the above corollary distributionally chaotic?
\end{question}

\begin{theorem}
There exists a mixing absolutely Ces\`aro bounded $C_0$-semigroup
$\cT$ on $L^p(1,\infty)$ for $1 \leq p < \infty$.
\end{theorem}

\begin{proof}
Let $0 < \eps < 1$ and consider the weighted translation semigroup $(T_t)$
on $L^p(1,\infty)$ defined by
$$
T_tf(x):= \Big(\frac{x+t}{x}\Big)^{\frac{1-\varepsilon}{p}}f(x+t).
$$
Given $f \in L^p(1,\infty)$ with $\|f\|_p = 1$, and $b > 0$, we have that
\begin{align*}
\int_0^b \|T_tf\|^p dt
&= \int_0^b \Big(\int_1^\infty \Big(\frac{x+t}{x}\Big)^{1-\eps}|f(x+t)|^p dx\Big)dt\\
&= \int_0^b \Big(\int_{1+t}^\infty \Big(\frac{x}{x-t}\Big)^{1-\eps}|f(x)|^p dx\Big)dt\\
&= \int_1^\infty x^{1-\eps}|f(x)|^p \int_0^{\min\{x-1,b\}}\Big(\frac{1}{x-t}\Big)^{1-\eps}dt\,dx\\
&\leq \int_1^{1+2b} x^{1-\eps}|f(x)|^p \int_0^{x-1} \Big(\frac{1}{x-t}\Big)^{1-\eps}dt\,dx\\
&\ \ \ \ \ \
 + \int_{1+2b}^\infty |f(x)|^p \int_0^b \Big(\frac{x}{x-t}\Big)^{1-\eps}dt\,dx\\
&\leq \int_{1}^{1+2b} \frac{x-1}{\eps}|f(x)|^p dx + 2b
 \leq \Big(2 + \frac{2}{\eps}\Big) b.
\end{align*}
So,
$$
\Big(\frac{1}{b} \int_0^b \|T_tf\|dt\Big)^p
\leq \frac{1}{b} \int_0^b \|T_tf\|^p dt
\leq 2 + \frac{2}{\eps}\cdot
$$
Thus, $(T_t)$ is an absolutely Ces\`aro bounded $C_0$-semigroup.
Let us now see that $(T_t)$ is mixing. If $v(x):= (\frac{1}{x})^{1-\eps}$,
then $T_t$ can be rewritten as
$$
T_tf(x) = \Big(\frac{v(x)}{v(x+t)}\Big)^{\frac{1}{p}} f(x+t).
$$
Since $v$ is an admissible weight function, the translation semigroup
defined as
$$
\tau_tf(x):= f(x+t)
$$
is a $C_0$-semigroup on $L_v^p(1,\infty)$. Moreover, $T_t$ on $L^p(1,\infty)$
and $\tau_t$ on $L_v^p(1,\infty)$ are topologically conjugate. Since
$(\tau_t)$ is mixing because $\lim_{x \to \infty} v(x) = 0$
\cite[Example~7.10(b)]{GEP11}, we conclude that $(T_t)$ is also mixing.
\end{proof}

\begin{corollary}
There exists a mixing not absolutely mean irregular $C_0$-semigroup $\cT$
on $L^p(1,\infty)$ for $1 \leq p < \infty$.
\end{corollary}

\begin{theorem}\label{DCNot1}
There exists a forward translation $C_0$-semigroup $\cT$ on $L_v^p(\R_+)$
which is distributionally chaotic and not mean Li-Yorke chaotic.
\end{theorem}

\begin{proof}
It was proved in \cite[Theorem 25]{BBPW} that there is a sequence
$w = (w_n)_{n \in \N}$ of positive weights such that the unilateral
weighted forward shift
$$
F_w : (x_1,x_2,\ldots) \in \ell^p(\N) \mapsto
      (0,w_1x_1,w_2x_2,\ldots) \in \ell^p(\N)
$$
is distributionally chaotic and satisfies
$$
\lim_{N \to \infty} \frac{1}{N} \sum_{j=1}^N \|(F_w)^jx\| = \infty
  \ \ \mbox{  for all } x \in \ell^p(\N) \backslash \{0\}.
$$
By using conjugacy, we see that there is a sequence $v' = (v_n)_{n \in \N}$
of positive weights such that the unweighted forward shift on
$\ell^p(v',\N)$ is distributionally chaotic and not mean Li-Yorke chaotic.
Now, if we consider as admissible weight function $v$ the polygonal formed
by the sequence $v'$, then the forward translation $C_0$-semigroup
on $L_v^p(\R_+)$ is distributionally chaotic and not mean Li-Yorke chaotic.
\end{proof}

\begin{remark}
The example in Theorem~\ref{t_bil_cami} can be easily adapted to the
semigroup setting in order to construct a completely absolutely mean
irregular $C_0$-semigroup $\cT$. Indeed, the translation semigroup on
$L^p_v(\R)$ does the job if we fix $v(k) = v_k$, $k \in \Z$, where
$(v_k)_{k \in \Z}$ is the sequence of weights of the example in
Theorem~\ref{t_bil_cami}, and we set $v(x)=v(k)$, $x\in\; ]k-1,k]$, $k\in\Z$.
\end{remark}


\section*{Acknowledgements}

This work was partially done on a visit of the first author to the
\emph{Institut Universitari de Matemàtica Pura i Aplicada} at
\emph{Universitat Polit\`ecnica de Val\`encia}, and he is very grateful
for the hospitality and support.
The first author was partially supported by CNPq (Brazil). 
The second and third authors were supported by MINECO and FEDER,
Project MTM2016-75963-P.  
The third author was also supported by Generalitat Valenciana,
Project PROMETEO/2017/102.


\small

\smallskip
\noindent N. C. Bernardes Jr.\\
\textit{ Departamento de Matem\'atica Aplicada, Instituto de Matem\'atica,
Universidade Federal do Rio de Janeiro, Caixa Postal 68530, Rio de Janeiro,
RJ 21945-970, Brazil.}\\
\textit{ e-mail:} ncbernardesjr@gmail.com

\smallskip
\noindent A. Bonilla\\
\textit{ Departamento de An\'alisis Matem\'atico, Universidad de la Laguna,
38271, La Laguna (Tenerife), Spain.}\\
\textit{ e-mail:} abonilla@ull.es

\smallskip
\noindent A. Peris\\
\textit{ Institut Universitari de Matemàtica Pura i Aplicada, Universitat Polit\`ecnica de Val\`encia, Edifici 8E, Acces F, 4a planta, 46022 Val\`encia, Spain.}\\
\textit{ e-mail:} aperis@mat.upv.es


\begin{thebibliography}{99}

\bibitem{ABMP}
    A. A. Albanese, X. Barrachina, E. M. Mangino and A. Peris, 
    \textit{Distributional chaos for strongly continuous semigroups of
    operators},
    Commun.\ Pure Appl.\ Anal.\ \textbf{12} (2013), no.\ 5, 2069--2082.

\bibitem{BC12}
    X. Barrachina and J. A. Conejero, 
    \textit{Devaney chaos and distributional chaos in the solution of
    certain partial differential equations},
    Abstr.\ Appl.\ Anal.\ 2012, Art.\ ID 457019, 11 pp.
    
\bibitem{BP12}
    X. Barrachina and A. Peris, 
    \textit{Distributionally chaotic translation semigroups}, 
    J. Difference Equ.\ Appl.\ \textbf{18} (2012), no.\ 4, 751--761.

\bibitem{BaGr06}
    F. Bayart and S. Grivaux,
    \textit{Frequently hypercyclic operators},
    Trans.\ Amer.\ Math.\ Soc.\ \textbf{358} (2006), no.\ 11, 5083--5117.

\bibitem{BM}
    F. Bayart and \'E. Matheron,
    \textit{Dynamics of Linear Operators},
    Cambridge University Press, Cambridge, 2009.

\bibitem{BR}
    F. Bayart, F. and I. Z. Ruzsa,
    \textit{Difference sets and frequently hypercyclic weighted shifts},
    Ergodic Theory Dynam.\ Systems \textbf{35} (2015), no.\ 3, 691--709.
    
\bibitem{beltran14} 
    M. J. Beltr\'an-Meneu, 
    \textit{Operators on weighted spaces of holomorphic functions}, 
    PhD Thesis, Universitat Polit\`ecnica de Val\`encia, 2014.

\bibitem{BBMP11}
    T. Berm\'{u}dez, A. Bonilla, F. Mart\'{\i}nez-Gim\'{e}nez and A. Peris,
    \textit{Li-Yorke and distributionally chaotic operators},
    J. Math.\ Anal.\ Appl.\ \textbf{373} (2011), no.\ 1, 83--93.

\bibitem{BermBMP}
    T. Berm\'{u}dez, A. Bonilla, V. M\"uller and A. Peris,
    \textit{Ces\`aro bounded operators in Banach spaces},
    J. Anal.\ Math.\ (to appear).
    
\bibitem{BB16}
     L. Bernal-González and A. Bonilla, 
     \textit{Order of growth of distributionally irregular entire functions
     for the differentiation operator},
     Complex Var.\ Elliptic Equ.\ \textbf{61} (2016), no. 8, 1176--1186.

\bibitem{BBMP}
    N. C. Bernardes Jr., A. Bonilla, V. M\"uller and A. Peris,
    \textit{Distributional chaos for linear operators},
    J. Funct.\ Anal.\ \textbf{265} (2013), no.\ 9, 2143--2163.

\bibitem{BBMP2}
    N. C. Bernardes Jr., A. Bonilla, V. M\"uller and A. Peris,
    \textit{Li-Yorke chaos in linear dynamics},
    Ergodic Theory Dynam.\ Systems \textbf{35} (2015), no.\ 6, 1723--1745.

\bibitem{BBPW}
    N. C. Bernardes Jr., A. Bonilla, A. Peris and X. Wu,
    \textit{Distributional chaos for operators on Banach spaces},  
    J. Math.\ Anal.\ Appl.\ \textbf{459} (2018), no.\ 2, 797--821.
    
\bibitem{BPR17} 
    N. C. Bernardes Jr., A. Peris and F. Rodenas, 
    \textit{Set-valued chaos in linear dynamics}, 
    Integral Equations Operator Theory \textbf{88} (2017), no.\ 4, 451--463.

\bibitem{BMPP}
    J. B\`es, Q. Menet, A. Peris and Y. Puig,
    \textit{Recurrence properties of hypercyclic operators},
    Math.\ Ann.\ \textbf{366} (2016), no.\ 1-2, 545--572.

\bibitem{BoGE07}
    A. Bonilla and K.-G. Grosse-Erdmann,
    \textit{Frequently hypercyclic operators and vectors},
    Ergodic Theory Dynam.\ Systems \textbf{27} (2007), no.\ 2, 383--404.
    Erratum: Ergodic Theory Dynam.\ Systems \textbf{29} (2009), no.\ 6, 1993--1994.
    
\bibitem{CLMP}
    J. A. Conejero, C. Lizama, M. Murillo-Arcila and A. Peris, 
    \textit{Linear dynamics of semigroups generated by differential operators}, 
    Open Math.\ \textbf{15} (2017), 745--767.  
    
\bibitem{CMP} 
    J. A. Conejero, V. Müller and A. Peris,  
    \textit{Hypercyclic behaviour of operators in a hypercyclic $C_0$-semigroup},
     J. Funct.\ Anal.\ \textbf{244} (2007), no.\ 1, 342--348.
    
\bibitem{CRT}
    J. A. Conejero, F. Rodenas and M. Trujillo, 
    \textit{Chaos for the hyperbolic bioheat equation}, 
    Discrete Contin.\ Dyn.\ Syst.\ \textbf{35} (2015), no.\ 2, 653--668.

\bibitem{D}
    T. Downarowicz,
    \textit{Positive topological entropy implies chaos DC2},
    Proc.\ Amer.\ Math.\ Soc.\ \textbf{142} (2014), no.\ 1, 137--149.

\bibitem{EN} K.-J. Engel and R. Nagel,
    \textit{One-Parameter Semigroups for Linear Evolution Equations},
    Springer-Verlag, New York - Berlin, 2000.

\bibitem{F}
    N. S. Feldman,
    \textit{Hypercyclicity and supercyclicity for invertible bilateral
    weighted shifts},
    Proc.\ Amer.\ Math.\ Soc.\ \textbf{131} (2003), no.\ 2, 479--485.

\bibitem{F-KOS}
    M. Fory\'s-Krawiec, P. Oprocha and M. \v{S}tef\'ankov\'a,
    \textit{Distributionally chaotic systems of type 2 and rigidity},
    J. Math.\ Anal.\ Appl.\ \textbf{452} (2017), no.\ 1, 659--672.

\bibitem{GARJIN}
    F. Garcia-Ramos and L. Jin,
    \textit{Mean proximality and mean Li-Yorke chaos},
    Proc.\ Amer.\ Math.\ Soc.\ \textbf{145} (2017), no.\ 7, 2959--2969.

\bibitem{GM}
    S. Grivaux and  \'E. Matheron,
    \textit{Invariant measures for frequently hypercyclic operators},
    Adv.\ Math.\ \textbf{265} (2014), 371--427.

\bibitem{GEP11}
    K.-G. Grosse-Erdmann and A. Peris Manguillot,
    \textit{Linear Chaos}, Springer, London, 2011.

\bibitem{HCC09}
    B. Hou, P. Cui and Y. Cao,
    \textit{Chaos for Cowen-Douglas operators},
    Proc.\ Amer.\ Math.\ Soc.\ \textbf{138} (2010), no.\ 3, 929--936.

\bibitem{HouLuo15}
    B. Hou and L. Luo,
    \textit{Some remarks on distributional chaos for bounded linear operators},
    Turkish J. Math.\ \textbf{39} (2015), no.\ 2, 251--258.

\bibitem{HLY} W. Huang, J. Li and X. Ye,
    \textit{Stable sets and mean Li-Yorke chaos in positive entropy systems},
    J. Funct.\ Anal.\ \textbf{266} (2014), no.\ 6, 3377--3394.

\bibitem{Leon}
    F. Le\'on-Saavedra,
    \textit{Operators with hypercyclic Ces\`aro means},
    Studia Math.\ \textbf{152} (2002), no.\ 3, 201--215.

\bibitem{LTY} J. Li, S. Tu and X. Ye,
    \textit{Mean equicontinuity and mean sensitivity},
    Ergodic Theory Dynam.\ Systems \textbf{35} (2015), no.\ 8, 2587--2612.

\bibitem{MOP09}
    F. Mart\'{\i}nez-Gim\'{e}nez, P. Oprocha and A. Peris,
    \textit{Distributional chaos for backward shifts},
    J. Math.\ Anal.\ Appl.\ \textbf{351} (2009), no.\ 2, 607--615.

\bibitem{MOP13}
    F. Mart\'{\i}nez-Gim\'{e}nez, P. Oprocha and A. Peris,
    \textit{Distributional chaos for operators with full scrambled sets},
    Math.\ Z. \textbf{274} (2013), no.\ 1-2, 603--612.

\bibitem{M}
    Q. Menet,
    \textit{Linear chaos and frequent hypercyclicity },
    Trans.\ Amer.\ Math.\ Soc.\ \textbf{369} (2017), no.\ 7, 4977--4994.

\bibitem{MV}
    V. M\"uller and J. Vr\v{s}ovsk\'y,
    \textit{Orbits of linear operators tending to infinity},
    Rocky Mountain J. Math.\ \textbf{39} (2009), no.\ 1, 219--230.

\bibitem{SS}
    J. Sm\'ital and M. \v{S}tef\'ankov\'a,
    \textit{Distributional chaos for triangular maps},
    Chaos Solitons Fractals \textbf{21} (2004), no.\ 5, 1125--1128.
    
\bibitem{W14}
    X. Wu, 
    \textit{Li-Yorke chaos of translation semigroups}, 
    J. Difference Equ.\ Appl.\ \textbf{20} (2014), no. 1, 49--57.

\bibitem{WWC}    
    X. Wu, L. Wang and G. Chen, 
    \textit{Weighted backward shift operators with invariant distributionally
    scrambled subsets},
    Ann.\ Funct.\ Anal.\ \textbf{8} (2017), no.\ 2, 199--210.

\bibitem{WOC} 
    X. Wu, P. Oprocha and G. Chen,
    \textit{On various definitions of shadowing with average error in tracing},
    Nonlinearity \textbf{29} (2016), no.\ 7, 1942--1972.
    
\bibitem{YY16}
    Z. Yin and Q. Yang,
    \textit{Generic distributional chaos and principal measure in linear
    dynamics},
    Ann.\ Polon.\ Math.\ \textbf{118} (2016), no.\ 1, 71--94.
    
\bibitem{YY17}
    Z. Yin and Q. Yang,
    \textit{Distributionally n-scrambled set for weighted shift operators}, 
    J. Dyn.\ Control Syst.\ \textbf{23} (2017), no.\ 4, 693--708.
    
\bibitem{YY18} 
    Z. Yin and Q. Yang, 
    \textit{Distributionally n-chaotic dynamics for linear operators}, 
    Rev.\ Mat.\ Complut.\ \textbf{31} (2018), no.\ 1, 111--129.

\end{thebibliography}
\end{document}